\documentclass[12pt]{amsart}

\usepackage{amsfonts,amssymb}
\usepackage{hyperref, graphicx}
\usepackage{epsfig}
\usepackage{latexsym}
\usepackage{amsmath,amsthm}
\usepackage{mathrsfs}
\usepackage[all,cmtip]{xy}
\usepackage{comment}
\usepackage[dvipsnames]{xcolor}

\usepackage{bbm}
\theoremstyle{plain}
\newtheorem{theorem}{Theorem}[section]
\newtheorem{lemma}[theorem]{Lemma}
\newtheorem{definition}[theorem]{Definition}

\newtheorem{cor}[theorem]{Corollary}
\newtheorem{remark}[theorem]{Remark}
\newtheorem{example}[theorem]{Example}

\numberwithin{equation}{section}

\begin{document}
\baselineskip=15.5pt

\title{Notes on fundamental groupoid schemes}
\author{Pavan~Adroja}
\address{Department of Mathematics, IIT Gandhinagar,
 Near Village Palaj, Gandhinagar - 382355, India}
 \email{pavan.a@iitgn.ac.in, adrojapavan@gmail.com}
\author{Sanjay~Amrutiya}
\address{Department of Mathematics, IIT Gandhinagar,
 Near Village Palaj, Gandhinagar - 382355, India}
 \email{samrutiya@iitgn.ac.in}
\thanks{The research work of PA is financially supported by CSIR-UGC fellowship under 
the grant no. 191620092279. SA is supported by the SERB-DST under project no. CRG/2023/000477.}
\date{}

\begin{abstract}
In this article, we study the various fundamental groupoid schemes corresponding to Tannakian categories of certain types of vector bundles. We compute fundamental groupoid scheme of anisotropic conic, Klein bottle and abelian varieties. Additionally, we study the relation among various fundamental groupoid schemes by considering their representations.
\end{abstract}

\maketitle
\section{Introduction}
Let $X$ be a proper, connected, and reduced scheme over a perfect field $k$. When $X$ admits a $k$-rational point $x$, Madhav Nori \cite{No} defined a fundamental group scheme $ \pi^{\rm N}(X, x)$ associated with the neutral Tannakian category of essentially finite bundles, now known as the Nori fundamental group scheme.  

In subsequent developments, Langer \cite{La11, La12} introduced the $S$-fundamental group scheme, which is Tannaka dual to the category of numerically flat bundles. Otabe \cite{Ot} extended Nori's framework in characteristic zero by introducing the category of semi-finite bundles, which strictly contains the category of finite bundles. The affine group scheme $\pi^{\rm EN}(X, x)$ corresponding to the neutral Tannakian category of semi-finite bundles, equipped with the fiber functor $x^*$, is called the extended Nori fundamental group scheme.

The existence of these fundamental group schemes relies on a $k$-rational point of a given $k$-scheme $X$. In particular, a $k$-rational point $x$ leads to the neutral fiber functor $x^*$ defined by the pullback. If $X$ does not have a rational point, then we cannot get a natural fibre functor, but one can get a functor that takes value in the category of quasi-coherent sheaves over some $k$-scheme $S$. By Tannaka duality, this leads to an affine $k$-groupoid scheme that acts transitively on $S$.

In this article, we recall the notion of groupoid schemes, their properties, and their representations. We then provide an overview of Tannakian formalism by following \cite{PD}. In Section \ref{Section-Funda grpoid}, we study groupoid versions of certain Tannakian categories. In particular, we compute fundamental groupoid schemes of anisotropic conic, Klein bottle and abelian varieties. Furthermore, we describe the relation among various fundamental groupoid schemes by considering their representations.


\section{Preliminaries}
A groupoid is a category in which every morphism is an isomorphism. A classical example is the topological fundamental groupoid, where the objects are the points of a topological space, and the morphisms are homotopy classes of paths between two points. One can extend this notion to groupoid schemes by considering the functor of points, which takes values in a groupoid. To formalize this, we first introduce the precise definition of a groupoid scheme. Let $k$ be a field and $S$ be any $k$-scheme.

\begin{definition}\label{def-grpo}
\rm An \emph{affine $k$-groupoid scheme} $G$ acting transitively on a $k$-scheme $S$ is a $k$-scheme $G$ with
\begin{itemize}
	\item a faithfully flat affine morphism $(t,s):G \longrightarrow S\times_k S$ of $S\times_k S$-schemes,
	\item product morphism $m: G\times_{^sS^{\,t}} G \longrightarrow G$ of $S\times_k S$-schemes,
	\item unit element morphism $e:S\longrightarrow G$ over $S\times_k S$, where $S$ is over $S\times_k S$ by a diagonal morphism $\Delta: S\longrightarrow S\times_k S$, and
	\item inverse element morphism $i:G \longrightarrow G$ of $S\times_k S$-scheme such that $s\circ i = t$ and $t\circ i=s$ 
\end{itemize}
which satisfies the following axioms:
\begin{enumerate}
	\item Associativity: $m\circ (m\times id) = m\circ (id \times m)$.
	\item Identity: $m\circ (e\times id) = m\circ (id \times e)$.
	\item Inverse: $m\circ (i\times id) = e\circ s$.
\end{enumerate}
\end{definition}

Let us interpret the above definition through the functor of points perspective. Let $(y,x): T\rightarrow S\times S$ be any morphisms. Consider the category $\mathcal{G}(T)$ whose objects are $S(T)$ and morphisms are $G_{y,x}(T)$ for any two object $x,y\in S(T)$; where $G_{y,x}$ defined by the following diagram:
\[
\xymatrix{    
          G_{y,x} \ar[r] \ar[d] & G \ar[d]^{(t,s)} \\
          T \ar[r]^{(y,x)\,\,\,\,\,\,\,\,}   & S\times_k S \\
         } 
\]
The category $\mathcal{G}(T)$ is a groupoid. We will use the above notations throughout this section.

\begin{example}
\rm Consider $G=S\times_k S$ with the map $(pr_1, pr_2): S\times_k S \rightarrow S\times_k S$. Then, this forms a groupoid scheme acting on a $k$-scheme $S$; where $pr_1$ and $pr_2$ denote the first and second projection maps, respectively.
\end{example}

\begin{example}
\rm If $G$ is an affine group scheme over a field $k$ with the structure map $f: G\rightarrow \mathrm{Spec\,}k$. Then, $(t,s)=(f,f)$ gives a natural structure of $k$-groupoid scheme acting transitively on $\mathrm{Spec\,}k$. In this sense, a groupoid scheme is a generalization of a group scheme.
\end{example}

Observe that, if $t=s$ in Definition \ref{def-grpo}, then the morphism $(t,s):G\rightarrow S\times_k S$ factor through the diagonal morphism $\Delta:S\rightarrow S\times_k S$. In this case, $G$ becomes group scheme over $S$. Given a groupoid scheme $G$ over $S$, we can define the corresponding diagonal group scheme $G^{\Delta}$ over $S$ and, for any morphism of $k$-scheme $f:T\rightarrow S$, we have a groupoid scheme $G_T$ over $T$ by the following square diagrams. 
\[
\xymatrix{    
          G^{\Delta} \ar[r] \ar[d] & G \ar[d]^{(t,s)} & & G_{T} \ar[r] \ar[d] & G \ar[d]^{(t,s)} \\
          S \ar[r]^{\Delta \,\,\,\,\,\,\,}   & S\times_k S & & T\times_k T \ar[r]^{(f,f)}  & S\times_k S \\
         } 
\]

\subsection{Representations of groupoid scheme}
A representation of $G$ is a pair of quasi-coherent sheaf $V$ of $\mathcal{O}_S$-module together with an action $\rho$ of $G$ on $V$, which is compatible with multiplication and base change. That means, for any two objects $(y,x):T\rightarrow S\times_k S$, we have a map $\rho_{y,x}:G_{y,x}(T)\rightarrow \mathrm{Iso}_T(x^*V\rightarrow y^*V)$ which is compatible with multiplication and base change, and if $x=y$ then $\rho_{x,x}(\mathrm{Id}_x)=\mathrm{Id}_{x^*V}$. The category of quasi-coherent representations of $G$ on $S$ is denoted by $\mathbf{Rep}(S:G).$

Let $f: T'\rightarrow T$ be a morphism of $k$-schemes. We get two object $(y\circ f,x\circ f):T'\rightarrow S\times_k S$ by taking composition with $(y,x):T\rightarrow S\times_k S$. There is a natural map $f^\#: G_{y,x}(T) \rightarrow G_{y\circ f,x\circ f}(T')$ given by $g:T\rightarrow G_{y,x}$ maps to $(pr_1\circ g\circ f, \mathrm{id}_{T'}):T'\rightarrow G_{y\circ f,x\circ f}$; where $pr_1:G_{y,x}\rightarrow G$ is the first projection map. The following diagram gives compatibility with the base change $f:T'\rightarrow T$.
\[
\xymatrix{    
          G_{y,x}(T) \ar[rrr]^{\rho_{y,x}} \ar[d]_{f^\#} & & & \mathrm{Iso}_T(x^*V\rightarrow y^*V) \ar[d]^{f^*} \\
          G_{y\circ f,x\circ f}(T') \ar[rrr]_{\rho_{y\circ f,x\circ f}} & & & \mathrm{Iso}_{T'}(f^*x^*V\rightarrow f^*y^*V)\\
         } 
\]
From the above, we conclude that any representation comes from the representation $\rho_{t,s}: G_{t,s}(G)\rightarrow \mathrm{Iso}_G(s^*V\rightarrow t^*V)$ by the base change. Any given $g: T\rightarrow G_{y,x}$, take $f=pr_1\circ g: T \rightarrow G$ in the above diagram, then one can determined ${\rho_{y,x}}(g)$ by the base change map $f$. Here, the map $f^\#: G_{t,s}(G)\rightarrow G_{t\circ f,s\circ f}(T)=G_{y,x}(T)$ and $f^\#(\mathrm{id}_G, \mathrm{id}_G)=g$. That means, $\rho_{t,s}(G)(\mathrm{id}_G, \mathrm{id}_G):s^*V \simeq t^*V$ determines the representation of $G$ on $V$.

\begin{remark} \label{R1}
\rm The category $\mathbf{Rep}(S:G)$ satisfies the base-change property \cite[Remark 1.8]{PD}. i.e. for any $T\rightarrow S$, we have that $\mathbf{Rep}(S:G)$ is equivalent to $\mathbf{Rep}(T:G_T)$; where $G_T$ denote the base change of $G$. In particular, if $G=S\times_k S$ with $(t,s)=(pr_1,pr_2)$ are the first and second projection maps, then for any affine open set $U$ of $S$, we have $\mathbf{Rep}(S:S\times_k S)\simeq \mathbf{Rep}(U:U\times_k U)$ by the base change $U \hookrightarrow S$.
\end{remark}

\section{Tannakian category}
In this section, we recall some basic notions of the Tannakian category from \cite{PD}. Let $k$ be a field and $\mathcal{T}$ be a $k$-linear abelian rigid tensor category with $\mathrm{End}(\mathbbm{1})=k$.

\begin{definition}
\rm A \emph{fiber functor} of $\mathcal{T}$ on a $k$-scheme $S$ is an exact $k$-linear tensor functor $\omega: \mathcal{T}\longrightarrow \mathbf{QCoh}(S)$ provided with a functorial isomorphisms $\omega(X)\otimes \omega(Y)\longrightarrow \omega(X\otimes Y)$ for all $X,Y \in \mathrm{Ob}(\mathcal{T})$, which is compactible with associativity, commutativity and unity.
\end{definition}

Note that a fiber functor commutes with dual. The axioms on $\mathcal{T}$ force that $\omega$ takes value in locally free sheaves of finite rank (see \cite[1.9]{PD}).

\begin{definition}
\rm A category $\mathcal{T}$ is called \emph{Tannakian} over $k$ if $\mathcal{T}$ admits a fiber functor $\omega: \mathcal{T}\longrightarrow \mathbf{QCoh}(S)$ for a non-empty $k$-scheme $S$. If $S=\mathrm{Spec\,}k$, then we call $\mathcal{T}$ to be a \emph{neutral Tannakian category}.
\end{definition}

\begin{example}
\rm Let $G$ be an affine group scheme over a field $k$. Then, the category of linear representations of $G$ denote by $\mathbf{Rep}_k(G)$ with the forgetful functor on the category $\mathbf{Vec}_k$ of finite dimensional $k$-vector spaces forms a neutral Tannakian category over $k$. More generally, for any affine $k$-groupoid scheme $G$ acting transitively on $S$, the category $\mathbf{Rep}(S:G)$ with the forgetful functor on $\mathbf{QCoh}(S)$ is a Tannakian category.
\end{example}

Let $\omega:\mathcal{T}\rightarrow \mathbf{QCoh}(S)$ be a fiber functor. Let $pr_1,pr_2: S\times_k S\rightarrow S$ denote the first and second projection morphisms, respectively. Then, they induce a functor $pr_1^*, pr_2^*: \mathbf{QCoh}(S)\rightarrow \mathbf{QCoh}(S\times_k S)$ defined by the pullback. Define 
$$\mathbf{Aut}_k^{\otimes}(\omega):= \mathbf{Isom}_{S\times_k S}^{\otimes}(pr_2^*\omega, pr_1^*\omega).$$ This is a representable functor over $S\times_k S$, which yields a groupoid scheme acting on $S$ (see the following theorem by Deligne).

\begin{theorem}\cite[Theorem 1.12]{PD}\label{Tannaka}
Let $\mathcal{T}$ be a $k$-linear abelian rigid tensor category on a field $k$, and let $\omega:\mathcal{T}\rightarrow
 \mathbf{QCoh}(S)$ be a fiber functor of $\mathcal{T}$ on a $k$-scheme $S$. Then,
\begin{enumerate}
	\item the groupoid $\mathbf{Aut}_k^{\otimes}(\omega)$ is affine and faithfully flat on $S\times_k S$.
	\item $\omega$ induces an equivalence of $\mathcal{T}$ with the category $\mathbf{Rep}(S:\mathbf{Aut}_k^{\otimes}(\omega))$.
	\item let $G$ be an affine $k$-groupoid scheme acting transitively on $S$ and faithfully flat on $S\times_k S$. Let $\tilde{\omega}: \mathbf{Rep}(S:G)\rightarrow \mathbf{QCoh}(S)$ defined by forgetting an action of $G$. Then, we have $G \simeq \mathbf{Aut}_k^{\otimes}(\tilde{\omega})$.
\end{enumerate}
\end{theorem}

The above theorem provides an equivalence between an affine groupoid schemes acting transitively on $S$ with faithfully flat on $S\times_k S$ and Tannakian categories whose fiber functor takes value in $\mathbf{QCoh}(S)$. Now, let us recall a theorem from \cite{PD}, which tells about the existence of fiber functor when $k$ has characteristic zero.

\begin{definition}
\rm Let $X$ be an object in $\mathcal{T}$. The \emph{dimension} of $X$ is defined by $\mathrm{dim}(X):= ev \circ \delta$, where $\delta: \mathbbm{1}\longrightarrow X^{\vee}\otimes X$ and $ev:X^{\vee}\otimes X \longrightarrow \mathbbm{1}$.
\end{definition}


\begin{theorem}\cite[Theorem 7.1]{PD}
Let $\mathcal{T}$ be a $k$-linear abelian rigid tensor category on a field $k$ of characteristic zero. Then, the following conditions are equivalent:
\begin{enumerate}
	\item $\mathcal{T}$ is Tannakian.
	\item For all $X\in \mathrm{Ob}(\mathcal{T})$, we have $\mathrm{dim}(X)\in \mathbb{Z}_{\geq 0}$.
	\item For all $X\in \mathrm{Ob}(\mathcal{T})$, there exists an integer $n\geq 0$ such that $\wedge^n X=0$.
\end{enumerate}
\end{theorem}

As an application of the above theorem, we can say that any $k$-linear abelian rigid tensor full subcategory of the category of vector bundles forms Tannakian category. For a fiber functor, as mentioned in the proof of the above theorem, there exists a ring element $A$ in the induced category of $\mathcal{T}$ such that $V\otimes A\simeq A^{\mathrm{dim}(V)}$ for all $V \in \mathrm{Ob}(\mathcal{T})$. So, one can define a functor $V\mapsto \Gamma(V\otimes A)$ which takes value in the category of  $\Gamma(A)$-modules, which is equivalent to $\mathbf{QCoh}(\mathrm{Spec}\,\Gamma(A))$.

\begin{remark}
\rm If $k$ has positive characteristic, then the above result holds if the category $\mathcal{T}$ satisfies an extra condition of finiteness (see \cite{PE}).
\end{remark}

\subsection*{Tensor product of categories}
In \cite[\S 5, Page no. 142]{PD}, Deligne defined the tensor product of abelian categories. Let $\mathcal{T}_1$ and $\mathcal{T}_2$ be two $k$-linear abelian categories. The Deligne tensor product of $\mathcal{T}_1$ and $\mathcal{T}_2$ is a $k$-linear abelian category $\mathcal{T}$ with $k$-bilinear and exact-in-each-variable functor $\otimes_D:\mathcal{T}_1\times \mathcal{T}_2 \rightarrow \mathcal{T}$ which satisfies the universal property: for any $k$-linear abelian category $\mathcal{C}$ with a $k$-bilinear and exact-in-each-variable functor $T:\mathcal{T}_1\times \mathcal{T}_2 \rightarrow \mathcal{C}$, there exists a $k$-linear exact functor $T':\mathcal{T} \rightarrow \mathcal{C}$ such that the diagram 
\[
\xymatrix{
\mathcal{T}_1\times \mathcal{T}_2 \ar[d]_{T}  \ar[rr]^{\otimes_D}  & &  \mathcal{T} \ar[dll]^{T'} \\
\mathcal{C}  & &  \\
}
\]
commutes. 
We will denote $\mathcal{T}$ by $\mathcal{T}_1\otimes_D \mathcal{T}_2$. 

\begin{remark}\label{Rep(G_1*G_2)= Rep(G_1)*Rep(G_2)}
\rm By \cite[5.18, 6.21]{PD}, it follows that if \( G_1 \) and \( G_2 \) are affine groupoid schemes acting transitively on a \( k \)-scheme \( S \), then the category of representations \( \mathbf{Rep}(S:G_1 \times_k G_2) \) is equivalent to the Deligne tensor product of categories:  
\[
\mathbf{Rep}(S:G_1 \times_k G_2) \simeq \mathbf{Rep}(S:G_1) \otimes_D \mathbf{Rep}(S:G_2).
\]
\end{remark}


\section{Fundamental Groupoid Schemes}\label{Section-Funda grpoid}
In this section, we will define groupoid versions of fundamental group schemes by using the Tannakian duality. Let us recall some definitions and fix some notations. Throughout this section, let $X$ be a smooth scheme of finite type defined over a field $k$ of characteristic zero with $\mathrm{H}^0(X,\mathcal{O}_X)=k$.

\begin{definition}
\rm A vector bundle $V$ on $X$ is called \emph{finite} if there are two different polynomials $f\neq g \in \mathbb{N}[t]$ such that $f(V)\simeq g(V)$, where the multiplication is defined by tensor product and addition is given by direct sum.
\end{definition}

By \cite[Chapter I, Lemma 3.1]{No}, a line bundle is finite if and only if it is torsion. Recall that a line bundle $L$ on $X$ is called \emph{torsion} if there exists a positive integer $m$ such that  $L^{\otimes m} \simeq \mathcal{O}_X$.

\begin{definition}
\rm A vector bundle $V$ on $X$ is called \emph{unipotent} if there is a filtration 
$$V=V_0 \supset V_1 \supset V_2 \supset \cdots \supset V_{n-1} \supset V_n=0$$ such that each successive quotient $V_i/V_{i-1}\simeq \mathcal{O}_X$ for all $i$.
\end{definition}

\begin{definition}
\rm A vector bundle $V$ on $X$ is called \emph{semi--finite} if there is a filtration 
$$V=V_0 \supset V_1 \supset V_2 \supset \cdots \supset V_{n-1} \supset V_n=0$$ such that each successive quotient $V_i/V_{i-1}$ is indecomposable finite bundle for all $i$.
\end{definition}

\begin{definition}
\rm A vector bundle $V$ on $X$ is called \emph{numerically flat} if $V$ and $V^*$ are numerically effective bundles. That means, tautological line bundles $\mathcal{O}_{\mathbb{P}(V)}(1)$ and $\mathcal{O}_{\mathbb{P}(V^*)}(1)$ are numerically effective (nef).
\end{definition}

A vector bundle $V$ on $X$ is called \textit{Nori-semistable} if for any smooth projective curve $C$ and a non-constant morphism $f:C\rightarrow X,$ the pullback $f^*V$ is a semi-stable bundle of degree zero.

Consider the categories $\mathcal{C}^{\rm N}(X)$, $\mathcal{C}^{\rm uni}(X)$, $\mathcal{C}^{\rm EN}(X)$, and $\mathcal{C}^{\rm nf}(X)$, which are full subcategories of the category $\mathbf{QCoh}(X)$ of quasi-coherent sheaves on $X$ with objects consisting of finite, unipotent, semi-finite, and numerically flat bundles, respectively. These categories satisfy the following relations:
$$\mathcal{C}^{\rm N}(X), \mathcal{C}^{\rm uni}(X) \subset \mathcal{C}^{\rm EN}(X) \subset \mathcal{C}^{\rm nf}(X).$$

Each of these categories is a $k$-linear, abelian, rigid tensor category. If $X$ has a $k$-rational point $x$, then the category $\mathcal{C}^{\star}(X)$, with the fiber functor $x^*$, forms a neutral Tannakian category, where $\star = \mathrm{N}, \mathrm{uni}, \mathrm{EN}, \mathrm{nf}$. By Tannakian duality, these categories correspond to an affine group scheme $\pi^{\star}(X, x)$ for $\star = \mathrm{N}, \mathrm{uni}, \mathrm{EN}$, and $\pi^{\rm S}(X, x)$ for $\star = \mathrm{nf}$. The affine group schemes $\pi^{\rm N}(X,x)$, $\pi^{\rm uni}(X, x)$, $\pi^{\rm EN}(X, x)$, and $\pi^{\rm S}(X, x)$ are called the \textit{Nori fundamental group scheme}, \textit{unipotent fundamental group scheme}, \textit{extended Nori fundamental group scheme}, and \textit{$S$-fundamental group scheme}, respectively (see \cite{No, Ot, La11, La12} for more details).

Let $S$ be a $k$-scheme, and let $\omega: \mathcal{C}^{\star}(X) \rightarrow \mathbf{QCoh}(S)$ be a fiber functor. By Tannakian duality, this corresponds to a $k$-groupoid scheme $\Pi^{\star}(X, \omega)$ acting on $S$. The groupoid schemes $\Pi^{\rm N}(X, \omega)$, $\Pi^{\rm uni}(X, \omega)$, $\Pi^{\rm EN}(X, \omega)$, and $\Pi^{\rm S}(X, \omega)$ are called the \textit{Nori fundamental groupoid scheme}, \textit{unipotent fundamental groupoid scheme}, \textit{extended Nori fundamental groupoid scheme}, and \textit{$S$-fundamental groupoid scheme}, respectively.

\subsection{Fundamental groupoid scheme of an anisotropic conic}
Let $X$ be a Riemann surface and $\sigma : X \rightarrow X$ be an anti-holomorphic involution. The pair $(X, \sigma)$ determines the real curve. If $X=\mathbb{P}^1_{\mathbb{C}}$, then there are only two possibilities for $\sigma$ upto equivalents (see \cite[Chapter-II, Exercise-4.7(e)]{Ha}).
\begin{enumerate}
	\item Let $\sigma_1:\mathbb{P}^1_{\mathbb{C}}\rightarrow \mathbb{P}^1_{\mathbb{C}}$ given by $[z_1:z_2]\mapsto [\overline{z_1}:\overline{z_2}]$ be an anti-holomorphic involution. The pair $(\mathbb{P}^1_{\mathbb{C}}, \sigma_1)$ gives a real projective line $\mathbb{P}^1_{\mathbb{R}}$.
	\item Let $\sigma_2:\mathbb{P}^1_{\mathbb{C}}\rightarrow \mathbb{P}^1_{\mathbb{C}}$ given by $[z_1:z_2]\mapsto [-\overline{z_2}:\overline{z_1}]$ be an anti-holomorphic involution. The pair $(\mathbb{P}^1_{\mathbb{C}}, \sigma_2)$ gives an anisotropic conic given by $C:\{x_0^2+x_1^2+x_2^2=0\}$ contained in $\mathbb{P}^2_{\mathbb{R}}$. Note that $C$ does not have any $\mathbb{R}$-point.
\end{enumerate}

We want to find the various fundamental groupoid schemes for the above two real curves. By \cite[Theorem 5.3]{BN}, $\mathcal{C}^{\rm nf}(C)$ contains only trivial bundles on $C$ and so $\mathcal{C}^{\rm EN}(C)$, $\mathcal{C}^{\rm N}(C)$ and $\mathcal{C}^{\rm uni}(C)$. Consider a fiber functor 
$$\tau: \mathcal{C}^{\rm nf}(C) \longrightarrow \mathbf{QCoh}(C)$$
defined by an inclusion. Let $\overline{x}: \mathrm{Spec\,}\mathbb{C}\rightarrow C$ be a geometric point, then we have $\tau':=\overline{x}^*\circ \tau : \mathcal{C}^{\rm nf}(C) \rightarrow \mathbf{Vec}_{\mathbb{C}}$ a fiber functor taking value in $\mathbf{Vec}_{\mathbb{C}}$. By Tannaka duality, we get an $S$-fundamental $\mathbb{R}$-groupoid scheme $\Pi^{\rm S}(C, \tau')\simeq \mathrm{Spec\,}\mathbb{C} \times_\mathbb{R} \mathrm{Spec\,}\mathbb{C}$ which act transitively on $\mathrm{Spec\,}\mathbb{C}$ for the corresponding Tannakian category $(\mathcal{C}^{\rm nf}(C), \tau')$. By Remark \ref{R1}, we have 
$$\Pi^{\rm S}(C, \tau) \simeq C\times_{\mathbb{R}} C$$
and so $\Pi^{\rm EN}(C, \tau)\simeq \Pi^{\rm N}(C, \tau) \simeq \Pi^{\rm uni}(C, \tau) \simeq C\times_{\mathbb{R}} C$.

\begin{remark} \label{R2}
\rm Applying a similar computation for $\mathbb{P}^1_{\mathbb{R}}$, with a rational point $x: \mathrm{Spec} \, \mathbb{R} \rightarrow \mathbb{P}^1_{\mathbb{R}}$ and a fiber functor \(x^*\), yields $\Pi^{\mathrm{S}}(\mathbb{P}^1_{\mathbb{R}}, x^*) \simeq \mathrm{Spec} \, \mathbb{R} \times_{\mathbb{R}} \mathrm{Spec} \, \mathbb{R} \simeq \mathrm{Spec} \, \mathbb{R} $. Thus, we obtain the trivial group scheme over $\mathbb{R}$. More generally, if $X$ is $k$-scheme with a $k$-rational point $x$ such that the category $\mathcal{C}^{\rm nf}(X)$ contains only trivial bundles, then $\Pi^{\rm S}(X,x^*)\simeq \mathrm{Spec}\,k$.
\end{remark}

\begin{remark} 
\rm Suppose $X$ is a $k$-scheme such that the category $\mathcal{C}^{\rm nf}(X)$ contains only trivial bundles. If $X$ does not have any $k$-point, then $\Pi^{\mathrm{S}}(X, \overline{x}^*) \simeq \mathrm{Spec} \, \overline{k} \times_{k} \mathrm{Spec} \, \overline{k}$, where $\overline{x}: \mathcal{C}^{\rm nf}(X) \rightarrow \mathrm{Spec} \, \overline{k}$ is a geometric point. For example, $X$ is a nondegenerate conic \cite{BN09} or Brauer-Severi varieties \cite{SN12, SN24}. 
\end{remark}

\subsection{Fundamental groupoid scheme of a Klein bottle} A \textit{Klein bottle} $X$ is a geometrically connected smooth projective curve of genus one defined over $\mathbb{R}$, and it does not have any real point. In other words, it is a smooth elliptic curve $X_{\mathbb{C}}$ over $\mathbb{C}$ with an anti-holomorphic involution that does not have any fixed point.

Let $L$ be a torsion (or finite) line bundle on $X$. Therefore, there exists a smallest positive integer $n\in \mathbb{N}$ such that $L^{\otimes n}\simeq \mathcal{O}_X$ and for any $1\leq m \leq n-1$, we have $L^{\otimes m}\ncong \mathcal{O}_X$. Let $\mathcal{C}_L$ be the smallest $\mathbb{R}$-linear abelian rigid tensor full subcategory of $\mathbf{Coh}(X)$ containing a line bundle $L$. By following the construction \cite[Page no. 83]{No}, we have 
$$\mathrm{Ob}(\mathcal{C}_L)=\left\{\bigoplus_{i=1}^{i=k}V_i: V_i \in \{\mathcal{O}_X,L,L^{\otimes 2}, \ldots ,L^{\otimes n-1}\} \,\mathrm{and}\, k\in \mathbb{N}\right\}.$$

Consider the fiber functor
$$\omega:\mathcal{C}_L \longrightarrow \mathbf{Coh}(X_{\mathbb{C}}) \longrightarrow \mathbf{Vec}_{\mathbb{C}}$$
corresponding to the identity element $e:\mathrm{Spec}\,\mathbb{C}\rightarrow X_{\mathbb{C}}$ of $X_\mathbb{C}$. The pair $(\mathcal{C}_L,\omega)$ forms Tannakian category over $\mathbb{R}$. By Theorem \ref{Tannaka}, we have a groupoid scheme $(t,s):\Pi \rightarrow \mathrm{Spec}\,\mathbb{C} \times_{\mathbb{R}} \mathrm{Spec}\,\mathbb{C}$ acting transitively on $\mathrm{Spec}\,\mathbb{C} \times_{\mathbb{R}} \mathrm{Spec}\,\mathbb{C}$. Let $\Pi_t$ denote the groupoid scheme $\Pi$ with the base scheme $\mathrm{Spec}\,\mathbb{C}$ by considering $pr_1 \circ (t,s):\Pi \rightarrow \mathrm{Spec}\,\mathbb{C}$. We claim that $\Pi_t (\mathrm{Spec}\,\mathbb{C})= \frac{\mathbb{Z}}{n\mathbb{Z}} \times \mathrm{Gal}(\mathbb{C}|\mathbb{R})$. Consider the diagram 
\begin{equation} \label{Klein}
\xymatrix{    
		\mathrm{Spec}\,\mathbb{C} \ar[d]_{\mathrm{id}} \ar[r]^{a} & \Pi \ar[d]^{(t,s)} \\
        \mathrm{Spec}\,\mathbb{C} & \mathrm{Spec}\,\mathbb{C} \times_{\mathbb{R}} \mathrm{Spec}\,\mathbb{C} \ar[l]_{pr_1\,\,\,\,\,\,\,\,\,\,\,\,\,\,\,\,\,} \\
         } 
\end{equation}

Note that $\Pi_t (\mathrm{Spec}\,\mathbb{C}) =\{a: \text{the diagram (\ref{Klein}) commutes}\}$. We have only two morphisms, $(\mathrm{id},\mathrm{id}):\mathrm{Spec}\,\mathbb{C} \rightarrow \mathrm{Spec}\,\mathbb{C} \times_{\mathbb{R}} \mathrm{Spec}\,\mathbb{C}$ and $(\mathrm{id},\overline{\mathrm{id}}):\mathrm{Spec}\,\mathbb{C} \rightarrow \mathrm{Spec}\,\mathbb{C} \times_{\mathbb{R}} \mathrm{Spec}\,\mathbb{C}$ such that the diagram 
\begin{equation}
\xymatrix{    
		\mathrm{Spec}\,\mathbb{C} \ar[dr]_{\mathrm{id}} \ar[r] & \mathrm{Spec}\,\mathbb{C} \times_{\mathbb{R}} \mathrm{Spec}\,\mathbb{C} \ar[d]^{pr_1} \\
        & \mathrm{Spec}\,\mathbb{C} \\
         } 
\end{equation}
commutes; where $\overline{\mathrm{id}}$ denote the morphism corresponding to the map $\mathbb{C}\rightarrow \mathbb{C}$ defined by the complex conjugate. Thus, we partition the set $\Pi_t (\mathrm{Spec}\,\mathbb{C})$ as the following disjoint union: 
$$\Pi_t (\mathrm{Spec}\,\mathbb{C}) = \{a: (t,s) \circ a = (\mathrm{id},\mathrm{id}) \} \cup \{a: (t,s) \circ a = (\mathrm{id},\overline{\mathrm{id}}) \}.$$

By definition, $\Pi$ is a functor 
$$\Pi: (\mathbf{Sch}|_{\mathrm{Spec}\,\mathbb{C} \times_{\mathbb{R}} \mathrm{Spec}\,\mathbb{C}}) \rightarrow \mathbf{Sets}$$
given by $(h:T\rightarrow \mathrm{Spec}\,\mathbb{C} \times_{\mathbb{R}} \mathrm{Spec}\,\mathbb{C}) \mapsto \mathbf{Isom}^{\otimes}_{T}(h^*pr_2^*\omega,h^*pr_1^*\omega)$; where $(\mathbf{Sch}|_{\mathrm{Spec}\mathbb{C} \times_{\mathbb{R}} \mathrm{Spec}\mathbb{C}})$ denote the category of schemes over $\mathrm{Spec}\,\mathbb{C} \times_{\mathbb{R}} \mathrm{Spec}\,\mathbb{C}$ and $\mathbf{Sets}$ denote the category of sets.

Therefore, we have 
$$\{a: (t,s) \circ a = (\mathrm{id},\mathrm{id}) \}=\Pi(\mathrm{id},\mathrm{id})= \frac{\mathbb{Z}}{n\mathbb{Z}}$$ and  
$$\{a: (t,s) \circ a = (\mathrm{id},\overline{\mathrm{id}}) \}=\Pi(\mathrm{id},\overline{\mathrm{id}})= \frac{\mathbb{Z}}{n\mathbb{Z}}.$$

 Hence, $\Pi_t (\mathrm{Spec}\,\mathbb{C})= \frac{\mathbb{Z}}{n\mathbb{Z}} \times \mathrm{Gal}(\mathbb{C}|\mathbb{R})$. Similarly, let $\mathcal{C}_{tor}$ be the smallest Tannakian category containing all torsion line bundles on $X$ having a fiber functor same as $\omega:\mathcal{C}_{tor} \rightarrow \mathbf{Vec}_{\mathbb{C}}$, then we have $\Pi_t (\mathrm{Spec}\,\mathbb{C})= \frac{\mathbb{Q}}{\mathbb{Z}} \times \mathrm{Gal}(\mathbb{C}|\mathbb{R})$ of the corresponding groupoid scheme $\Pi$.

\subsection{Fundamental groupoid scheme of an abelian variety}
Throughout this section, let $A$ be an abelian variety over a field $k$ of characteristic zero. Recall that, $\mathcal{C}^{\rm uni}(A),\mathcal{C}^{\rm N}(A), \mathcal{C}^{\rm EN}(A)$ are $k$-linear abelian rigid tensor category. To apply Tannaka duality, consider the fiber functor defined by inclusion in $\mathbf{QCoh}(A)$. 
$$i_{\rm uni},i_{\rm N}, i_{\rm EN}: \mathcal{C}^{\rm uni}(A),\mathcal{C}^{\rm N}(A), \mathcal{C}^{\rm EN}(A) \longrightarrow \mathbf{QCoh}(A)$$

The corresponding groupoid $k$-scheme denoted by $\Pi^{\rm uni}(A), \Pi^{\rm N}(A), \Pi^{\rm EN}(A)$ which acts transitively on $A$. By definition, we have
$$\Pi^{\rm uni}(A)= \mathbf{Aut}_k^{\otimes}(i_{\rm uni})= \mathbf{Isom}_{A\times_k A}^{\otimes}(pr_2^*i_{\rm uni},pr_1^*i_{\rm uni})$$
$$\Pi^{\rm N}(A)= \mathbf{Aut}_k^{\otimes}(i_{\rm N})= \mathbf{Isom}_{A\times_k A}^{\otimes}(pr_2^*i_{\rm N},pr_1^*i_{\rm N})$$
$$\Pi^{\rm EN}(A)= \mathbf{Aut}_k^{\otimes}(i_{\rm EN})= \mathbf{Isom}_{A\times_k A}^{\otimes}(pr_2^*i_{\rm EN},pr_1^*i_{\rm EN})$$
where $pr_1,pr_2: A\times_k A \longrightarrow A$ denote the first and second projection maps, respectively. We establish a relation among them.

\begin{theorem}\label{homo-nf-decomposition}
Let $A$ be an abelian variety defined over an algebraically closed field $k$ of characteristic zero. Let $E$ be a vector bundle on $A$. Then, the following are equivalent:
\begin{enumerate}
\item $E$ is homogeneous. That means, it is invariant $(\tau_x^*E \simeq E )$ under all translation maps $\tau_x: A\rightarrow A$ defined by $a \mapsto a+x$.
\item $E$ is an iterated extension of line bundles of degree zero.
\item $E\simeq \bigoplus_i (U_i\otimes L_i)$, where $U_i$'s are unipotent bundles and $L_i$'s are degree zero line bundles.
\item $E$ is numerically flat.
\end{enumerate}
\end{theorem}
\begin{proof}
(1) $\Longleftrightarrow$ (2) $\Longleftrightarrow$ (3) follows from \cite[Theorem 4.17]{Muk78} and \cite{Miy73}. Since unipotent bundles and degree zero line bundles are numerically flat, we have (3) $\Longrightarrow$ (4). The implication (4) $\Longrightarrow$ (2) follows from \cite[Corollary 5.5, Remark 5.6]{La12}. 
\end{proof}

\begin{remark}\label{semi-finite decomposition}
\rm By Theorem \ref{homo-nf-decomposition}, if $E$ is a semi-finite bundle then $E\simeq \bigoplus_i (U_i\otimes L_i)$, where $U_i$'s are unipotent bundles and $L_i$'s are torsion line bundles. Moreover, if $E$ is a finite bundle, then $E \simeq \bigoplus_i L_i$, where $L_i$'s are torsion line bundles.
\end{remark}

\begin{theorem}
Let $A$ be an abelian variety defined over an algebraically closed field $k$ of characteristic zero. Then, we have
$$\Pi^{\rm EN}(A) \simeq \Pi^{\rm uni}(A) \times_k \Pi^{\rm N}(A).$$
\end{theorem}
\begin{proof}
By Remark \ref{Rep(G_1*G_2)= Rep(G_1)*Rep(G_2)}, it is enough to prove that the category $\mathcal{C}^{\rm EN}(A)$ is equivalent to the category $\mathcal{C}^{\rm uni}(A)\otimes_D \mathcal{C}^{\rm N}(A)$. Consider the tensor product
$$\otimes_D: \mathcal{C}^{\rm uni}(A) \times \mathcal{C}^{\rm N}(A) \longrightarrow \mathcal{C}^{\rm uni}(A) \otimes_D \mathcal{C}^{\rm N}(A)$$
of the categories $\mathcal{C}^{\rm uni}(A)$ and $\mathcal{C}^{\rm N}(A)$. Consider the functor 
$$T:\mathcal{C}^{\rm uni}(A) \times \mathcal{C}^{\rm N}(A) \longrightarrow \mathcal{C}^{\rm EN}(A) $$
defined by $(U,F)\mapsto U\otimes F$ for $U\in \mathrm{Ob}(\mathcal{C}^{\rm uni}(A))$ and $F\in \mathrm{Ob}(\mathcal{C}^{\rm N}(A))$. Note that the functor $F$ is $k$-bilinear and exact-in-each-variable. By universal property of tensor product of categories, there exists a $k$-linear and exact functor 
$$T': \mathcal{C}^{\rm uni}(A) \otimes_D \mathcal{C}^{\rm N}(A) \longrightarrow \mathcal{C}^{\rm EN}(A)$$
such that the diagram 
\[
\xymatrix{
\mathcal{C}^{\rm uni}(A) \times \mathcal{C}^{\rm N}(A) \ar[d]_{T}  \ar[rr]^{\otimes_D}  & &  \mathcal{C}^{\rm uni}(A)\otimes_D \mathcal{C}^{\rm N}(A) \ar[dll]^{T'} \\
\mathcal{C}^{\rm EN}(A)  & &  \\
}
\]
commutes. By Remark \ref{semi-finite decomposition}, the functor $T'$ is essentially surjective. Since 
\begin{align*}
T'\Big(\mathrm{Hom}\big(\otimes_D(U_1, F_1), \otimes_D(U_2,F_2)\big)\Big) & = \mathrm{Hom}(U_1 \otimes F_1, U_2 \otimes F_2) \\
	& \simeq \mathrm{Hom}(U_1,U_2) \otimes_k \mathrm{Hom}(F_1,F_2)
\end{align*}
for all $U_1,U_2 \in \mathrm{Ob}(\mathcal{C}^{\rm uni}(A))$ and $F_1,F_2 \in \mathrm{Ob}(\mathcal{C}^{\rm N}(A))$, the functor $T'$ is full and faithful. The assertion follows as the fiber functor commutes with the functor $T'$.
\end{proof}

\begin{remark}  
\rm Let \( 0: \mathrm{Spec}\,k \to A \) denote the identity element of an abelian variety \( A \). Consider the fiber functor given by pullback:  
\[
0^*: \mathcal{C}^{\star}(A) \longrightarrow \mathbf{Vec}_k, \quad \text{for } \star = \mathrm{EN, uni, N}.
\]  
The functor \( 0^* \) is a neutral fiber functor, and we denote the corresponding affine group schemes by \( \pi^{\star}(A, 0) \). By Remark \ref{R1}, we obtain the following isomorphism:  
\[
\pi^{\rm EN}(A, 0) \simeq \pi^{\rm uni}(A, 0) \times_k \pi^{\rm N}(A, 0).
\]  
This provides an alternative proof of \cite[Theorem 3.3]{AA}.  
\end{remark}


\section{Representations}
The notion of closed immersion and faithfully flat morphism between affine group schemes are completely characterized by their representations (see \cite[Proposition 2.21]{DM}). In this section, we will extend this to an affine groupoid case. Consequently, we will describe the relation among various fundamental groupoid schemes by considering their representations. 

Let $k$ be a field and $\overline{k}$ denote its algebraic closure. Let $S=\mathrm{Spec\,}\overline{k}$ be an affine scheme over a field $k$. Let $(G_1,s_1,t_1,m_1, e_1, inv_1)$ and $(G_2,s_2,t_2,m_2, e_2, inv_2)$ are two affine $k$-groupoid schemes acting transitively on a $k$-scheme $S$. Therefore, we have the following morphisms for $i=1,2$:
\begin{itemize}
	\item a faithfully flat affine morphism $(t_i, s_i):G_i \longrightarrow S\times_k S$
	\item a morphism of $S\times_k S$-scheme $m_i:G_i \times_{^sS^t} G_i \longrightarrow G_i $, $e_i: S \longrightarrow G_i$ and $inv_i:G_{i}\longrightarrow G_{i}$.
\end{itemize}

Let $G_1=\mathrm{Spec\,}L_1$ and $G_2=\mathrm{Spec\,}L_2$; where $L_1$ and $L_2$ are faithfully flat $\overline{k}\otimes_k \overline{k}$-algebra as $(t_1,s_1)$ and $(t_2,s_2)$ are faithfully flat affine morphisms. A morphism between two groupoid schemes is defined in an obvious way.

\begin{definition}
\rm	A \emph{morphism $f:(G_1,s_1,t_1,m_1,e_1, inv_1) \rightarrow (G_2,s_2,t_2,m_2,e_2, inv_2)$ between $k$-groupoid schemes} acting on a $k$-scheme $S$ is a morphism of $k$-scheme $f:G_1\rightarrow G_2$ satisfying $s_2\circ f=s_1, t_2\circ f=t_1, f\circ m_1 = m_2 \circ (f,f), e_2=f\circ e_1$ and $f \circ inv_1 = inv_2 \circ f$. 
\end{definition}

Let $f:G_1 \rightarrow G_2$ be a morphism between $k$-groupoid schemes acting on a $k$-scheme $S$. Then,
it induces the morphism between the corresponding diagonal group schemes $f^{\Delta}=f\times \mathrm{id}_S :G_1^{\Delta}\rightarrow G_2^{\Delta}$ by the base change $\Delta:S\rightarrow S\times_k S$. Note that, $f$ is an isomorphism if and only if $f^{\Delta}$ is an isomorphism (see \cite[3.5.2]{PD}). 

\begin{definition}
\rm	A morphism $f:G_1\rightarrow G_2$ of groupoid schemes is called \emph{faithfully flat} (resp. \emph{closed immersion}) if it is faithfully flat (resp. closed immersion) as a $k$-scheme morphism.
\end{definition}

\begin{lemma} \label{Lff}
	A morphism $f:G_1 \rightarrow G_2$ is faithfully flat if and only if $f^{\Delta}:G_1^{\Delta}\rightarrow G_2^{\Delta}$ is faithfully flat.
\end{lemma}
\begin{proof}
	Suppose $f:G_1 \rightarrow G_2$ is a faithfully flat morphism. Then, the induced map $f^\#:L_2\rightarrow L_1$ is a faithfully flat algebra map. Hence, $L_1$ is a faithfully flat $L_2$-module. Now, consider the base change under the map $L_2 \rightarrow L_2 \otimes_{\overline{k}\otimes_k \overline{k}} \overline{k}$, then $L_1 \otimes_{L_2} (L_2 \otimes_{\overline{k}\otimes_k \overline{k}} \overline{k})\simeq L_1 \otimes_{\overline{k}\otimes_k \overline{k}} \overline{k}$ is a faithfully flat $L_2 \otimes_{\overline{k}\otimes_k \overline{k}} \overline{k}$-module. Hence, the induced morphism $f^{\Delta}:G_1^{\Delta}\rightarrow G_2^{\Delta}$ is faithfully flat. Conversely, assume that the map $f^\# \otimes \mathrm{id} :L_2 \otimes_{\overline{k}\otimes_k \overline{k}} \overline{k}\rightarrow L_1 \otimes_{\overline{k}\otimes_k \overline{k}} \overline{k}$ is faithfully flat. We have the following commutative diagram:
	\begin{equation} \label{eq- faithfully flat}
		\xymatrix{    
			L_2 \ar[d] \ar[r]^{f^\#} & L_1 \ar[d] \\
			L_2 \otimes_{\overline{k}\otimes_k \overline{k}} \overline{k} \ar[r]^{f^\# \otimes \mathrm{id}} & L_1 \otimes_{\overline{k}\otimes_k \overline{k}} \overline{k} \\
		} 
	\end{equation}
	Consider the natural inclusion $g: L_1 \otimes_{\overline{k}\otimes_k \overline{k}} \overline{k} \rightarrow L_1 \otimes_{\overline{k}\otimes_k \overline{k}} \overline{k} \otimes_k \overline{k} \simeq L_1$, which is faithfully flat. Since this map makes the above diagram commutative, it follows that $f^\#:L_2\rightarrow L_1$ is faithfully flat as it is a composition of faithfully flat maps.
\end{proof}

\begin{lemma} \label{Lci}
	A morphism $f:G_1 \rightarrow G_2$ is closed immersion if and only if $f^{\Delta}:G_1^{\Delta}\rightarrow G_2^{\Delta}$ is closed immersion.
\end{lemma}
\begin{proof}
	Since $S$ is an affine scheme, the map  $\Delta:S\rightarrow S\times_k S$ is closed immersion. The projection map $G_i^{\Delta}\rightarrow G_i$ is the base change of $\Delta:S\rightarrow S\times_k S$ by $(t_i,s_i):G_i\rightarrow S\times_k S$, so $G_i^{\Delta}\rightarrow G_i$ is closed immersion. If $f$ is closed immersion, then $f^{\Delta}$ is closed immersion by the diagram \eqref{eq- faithfully flat}. Let $f^{\Delta}$ be closed immersion. It is morphism of affine schemes so $f^\# \otimes \mathrm{id}:L_2 \otimes_{\overline{k}\otimes_k \overline{k}} \overline{k} \rightarrow L_1 \otimes_{\overline{k}\otimes_k \overline{k}} \overline{k}$ is surjective. This implies that the map $f^\#\otimes \mathrm{id} \otimes \mathrm{id}: L_2 \otimes_{\overline{k}\otimes_k \overline{k}} \overline{k} \otimes_k \overline{k} \rightarrow L_1 \otimes_{\overline{k}\otimes_k \overline{k}} \overline{k} \otimes_k \overline{k}$ is surjective and compactible with $f^\#$ by considering standard isomorphism $L_i \otimes_{\overline{k}\otimes_k \overline{k}} \overline{k} \otimes_k \overline{k} \rightarrow L_i$. Hence, $f^\#$ is surjective.
\end{proof}

Let $\omega_f: \mathbf{Rep}(S:G_2)\rightarrow \mathbf{Rep}(S:G_1)$ denote the induced tensor functor on the category of representations defined by $(V,\rho)\mapsto (V, f\circ \rho)$. 

\begin{theorem}
	Let $f:G_1\rightarrow G_2$ be a morphism between two affine $k$-groupoid scheme $G_1$ and $G_2$ acting transitively on a $k$-scheme $S=\mathrm{Spec\,}\overline{k}$. Let $\omega_f: \mathbf{Rep}(S:G_2)\rightarrow \mathbf{Rep}(S:G_1)$ be the induced functor. Then,
	\begin{enumerate}
		\item $f$ is faithfully flat morphism if and only if the functor $\omega_f$ is fully faithful and every subobject of $\omega_f(X)$, for $X \in \mathrm{Ob}(\mathbf{Rep}(S:G_2))$, is isomorphic to image of a subobject of $X$.
		\item $f$ is closed immersion if and only if every object of $\mathbf{Rep}(S:G_1)$ is isomorphic to a subquotient of an object $\omega_f(X)$, for $X\in \mathrm{Ob}(\mathbf{Rep}(S:G_2))$.
	\end{enumerate}
\end{theorem}
\begin{proof}
	(1) Let $f:G_1 \rightarrow G_2$ be a faithfully flat morphism. Then, there is an equivalence between the category $\mathbf{Rep}(S:G_2)$ and the full subcategory of $\mathbf{Rep}(S:G_1)$ consisting of representations of $G_1$ that factor through $f$. Hence, $\omega_f$ has the stated property. Conversely, suppose that $\omega_f$ has the stated properties. Note that any representation $(V,\rho)$ of $G_2$ is determined by the pair $(V, s_2^*V\simeq t_2^*V)$; where $V$ is a quasi-coherent $\mathcal{O}_S$-module. The image of $(V,\rho)$ under the functor $\omega_f$ is determined by $(V,f^*s_2^*V\simeq f^*t_2^*V)$ as $t_2\circ f=t_1$ and $s_2\circ f=s_1$. We have the following commutative diagram:
	\begin{equation} \label{B}
		\xymatrix{    
			G_1^{\Delta} \ar[d] \ar[r]^{f^{\Delta}} & G_2^{\Delta} \ar[d] \\
			G_1 \ar[r]^{f} & G_2  \\
		} 
	\end{equation}
	The above diagram gives the same condition on the functor $\omega_{f^{\Delta}}$, which implies that $f^{\Delta}$ is faithfully flat \cite[Proposition 2.21]{DM}. Hence, so is $f$ by Lemma \ref{Lff}.\\
	(2) Let $\mathcal{D}$ be the full subcategory of $\mathbf{Rep}(S:G_1)$ whose objects are subquotients of $\omega_f(X)$ for $X \in \mathrm{Ob}(\mathbf{Rep}(S:G_2))$. Clearly, $(\mathcal{D},\omega_1|_{\mathcal{D}})$ is Tannakian category. The sequence $\mathbf{Rep}(S:G_2)\rightarrow \mathcal{D} \rightarrow \mathbf{Rep}(S:G_1)$ induce a sequence on groupoid schemes $G_1 \rightarrow G_{\mathcal{D}} \rightarrow G_2$; where $G_{\mathcal{D}}$ denote the corresponding groupoid scheme acting transitively on $S$. By taking a diagonal group scheme, we have the same properties for the representation categories of diagonal group schemes. Then, the assertion follows by  \cite[Proposition 2.21]{DM} and Lemma \ref{Lci}.
\end{proof}

Let $X$ be a $k$-scheme of finite type with $\mathrm{H}^0(X,\mathcal{O}_X)=k$. Let $X_{\overline{k}}$ denote the base change and let $x:\mathrm{Spec}\,\overline{k}\rightarrow X_{\overline{k}}$ be a rational point of $X_{\overline{k}}$. Consider the fiber functor
$$\omega: \mathcal{C}^{\rm nf}(X),\mathcal{C}^{\rm EN}(X),\mathcal{C}^{\rm N}(X),\mathcal{C}^{\rm uni}(X)\longrightarrow \mathbf{QCoh}(X_{\overline{k}})\longrightarrow \mathbf{Vec}_{\overline{k}}$$
defined by the pullback along the composition map $\mathrm{Spec}\,\overline{k}\rightarrow X_{\overline{k}} \rightarrow X$ for each of these categories. 

\begin{cor}
	We have the following diagram of groupoid schemes, where all the maps are faithfully flat:
	\[
	\xymatrix{    
		\Pi^{\rm S}(X,\omega) \ar[r] & \Pi^{\rm EN}(X,\omega) \ar[r] \ar[d] & \Pi^{\rm N}(X,\omega) \\
		& \Pi^{\rm uni}(X,\omega) &\\
	} 
	\]
\end{cor}




\begin{thebibliography}{012345}
\bibitem{AA} P.~Adroja, S.~Amrutiya, \emph{On an extension of Nori and local fundamental group schemes}, (preprints) arXiv:2306.06861v2.
		          
\bibitem{BN} I.~Biswas, D.~S.~Nagaraj, \emph{Classification of real algebraic vector bundles over the real anisotropic conic}, International Journal of Mathematics 16.10 (2005).

\bibitem{BN09} I.~Biswas, D.~S.~Nagaraj, \emph{Vector bundles over a nondegenerate conic}, J. Aust. Math. Soc. {\bf 86} (2009), no.~2, 145--154.

\bibitem{PD} P.~Deligne, \emph{Catégories tannakiennes}. in The Grothendieck Festschrift, Vol. II, Progr. Math., vol. \textbf{87}, Birkhäuser, Boston, MA, 1990, p. 111-195.

\bibitem{DM} P.~Deligne,~J.~S.~Milne, Tannakian Categories, in:
         {\it Hodge cycles, motives and Shimura varieties} (by P.~Deligne,
         J.~S.~Milne, A.~Ogus and K.~Y.~Shih), pp. 101-228, \emph{Lecture 
         Notes in Mathematics,} 900, Springer-Verlag, Berlin-Heidelberg- 
         New York, 1982.
         
\bibitem{PE} P.~Etingof, S.~Gelaki, D.~Niksych, V.~Ostrik, \emph{Tensor Categories}, Mathematical Surveys and Monographs (2015).

\bibitem{Ha} R.~Hartshorne, Algebraic Geometry, Grad. Texts
         Maths., 52, Springer, New York-Heidelberg-Berlin, 1977.

\bibitem{La11} A.~Langer, \emph{On S-fundamental group scheme},  
		Ann. Inst. Fourier (Grenoble) \textbf{61} (2011), no. 5, 2077-2119.
		
\bibitem{La12} A.~Langer, \emph{On S-fundamental group scheme II},  
		J. Inst. Math. Jussieu \textbf{11} (2012), 835-854.

\bibitem{Miy73} M.~Miyanishi, \emph{Some remarks on algebraic homogeneous vector bundles}, in {\it Number theory, algebraic geometry and commutative algebra, in honor of Yasuo Akizuki}, pp. 71--93, Kinokuniya Book Store, Tokyo.
		
\bibitem{Muk78}	S.~Mukai, \emph{Semi-homogeneous vector bundles on an Abelian variety}, J. Math. Kyoto Univ. {\bf 18} (1978), no.~2, 239--272.

\bibitem{No} M.~V.~Nori, \emph{The fundamental group-scheme}, 
		Proc. Indian Acad. Sci. Math. Sci. \textbf{91} (1982), no. 2, p. 73-122. 
		
\bibitem{SN12} S.~Novakovi\'c, \emph{Absolutely split locally free sheaves on Brauer-Severi varieties of index two}, Bull. Sci. Math. {\bf 136} (2012), no.~4, 413--422.

\bibitem{SN24} S.~Novakovi\'c, \emph{Absolutely split locally free sheaves on proper $k$-schemes and Brauer-Severi varieties}, Bull. Sci. Math. {\bf 197} (2024), Paper No. 103494, 25 pp.
		
\bibitem{Ot} S.~ Otabe, \emph{An extension of Nori fundamental group}, 
		Comm. Algebra \textbf{45} (2017), 3422-3448.
		

\end{thebibliography}
\end{document}